\documentclass[12pt]
{amsart}
\usepackage{amsmath,amsthm,amsfonts,amssymb,mathrsfs,color} 

\newtheorem{theorem}{Theorem}
\newtheorem{lemma}{Lemma}
\newtheorem{proposition}{Proposition}
\newtheorem{corollary}{Corollary}
\theoremstyle{definition}

\newtheorem{example}{Example}
\newtheorem{remark}{Remark}

\numberwithin{equation}{section} \numberwithin{lemma}{section}
\numberwithin{theorem} {section} \numberwithin{remark} {section}
\numberwithin{proposition} {section} \numberwithin{corollary} {section}

\newcommand{\Sob}[2]{W^{1,#1}_0\def\next{#2}\ifx\next\empty\else(#2)\fi}
\newcommand{\SobpO}{\Sob{p}{\Omega}}
\newcommand{\Sobdual}[1]{W^{-1,p'}\def\next{#1}\ifx\next\empty\else(#1)\fi}
\newcommand{\SobdualO}{\Sobdual{\Omega}}
\let\mappaa\mapsto
\renewcommand{\mapsto}{\quad\mappaa\quad}
\newcommand{\FFF}{\mathcal F}
\newcommand\chiusura[1]{L^{#1,\infty}_0}
\newcommand{\dist}{\mathrm{dist}}

\newcommand{\fff}{\varphi}
\newcommand{\ee}{\varepsilon}

\newcommand{\ds}{\displaystyle}
\newcommand{\dx}{\,\mathrm d }
\newcommand{\Om}{\Omega}

\newcommand{\reale}{{\mathbb R}}
\newcommand{\naturale}{{\mathbb N}}
\newcommand{\Rn}{\reale^N}

\newcommand{\divergenza}{\mathop{\mathrm{div}}}

\makeatletter
\@namedef{subjclassname@2020}{%
  \textup{2020} Mathematics Subject Classification}
\makeatother

\begin{document}

\title[Noncoercive quasilinear equations]{Noncoercive quasilinear  elliptic operators with  singular lower order terms}
\author[Farroni, Greco, Moscariello, Zecca]{Fernando Farroni, Luigi Greco,\\Gioconda Moscariello and Gabriella Zecca}
\address{Dipartimento di Ingegneria Elettrica e delle Tecnologie dell'Informazione, Universit\`a degli Studi di Napoli ``Federico~II'', Via Claudio~21, 80125 Napoli, Italy}\email{luigreco@unina.it}
\address{Dipartimento di Matematica e Applicazioni ``R. Caccioppoli'', Universit\`a degli Studi di Napoli ``Federico~II'', Via Cintia, 80126 Napoli, Italy}\email{fernando.farroni@unina.it} \email{gmoscari@unina.it} \email{g.zecca@unina.it}
\thanks{The authors are members of Gruppo Nazionale per l'Analisi Matematica, la Probabilit\`a e le loro Applicazioni (GNAMPA) of INdAM. The research of G.M. has been partially supported by the National Research Project PRIN \lq\lq Gradient flows, Optimal Transport and Metric Measure Structures'', code~2017TEXA3H}



\maketitle

\begin{abstract}
We consider a family of quasilinear second order elliptic differential operators which are not coercive and are defined by functions in Marcinkiewicz spaces. We prove the existence of a solution to the corresponding Dirichlet problem. The associated 
obstacle problem is also solved. Finally, we show higher integrability of a solution to the Dirichlet problem when the datum is more regular.
\end{abstract}
%
%
%

\section{Introduction}

Given a bounded domain 
$\Omega$ of  $\mathbb R^N$, $N\ge  2$, we consider
\[A\colon (x,u,\xi)\in \Omega\times \reale\times \mathbb R^N\mapsto \mathbb R^N\]
a Carath\'eodory vector field (i.e.~measurable in $x\in \Omega$ and continuous in $(u,\xi)\in \reale\times \mathbb R^N$) 
satisfying for a.e.\ $x\in\Om$ and  for every $u\in\reale$ and $\xi,\eta\in\Rn$ the following structural conditions:
\begin{equation}\label{coercivita}
\langle A(x,u,\xi),\xi\rangle\geqslant \alpha|\xi|^p-\big(b(x)\,|u|\big)^p-\fff(x)^{p}
\end{equation}
\begin{equation}\label{limitatezza}
|A(x,u,\xi)|\leqslant \beta|\xi|^{p-1}+\big(b(x)\,|u|\big)^{p-1}+\fff(x)^{p-1}
\end{equation}
\begin{equation}\label{monotonia}
\langle A(x,u,\xi)- A (x,u,\eta ), \xi -\eta \rangle>0\qquad \text{for }\xi\ne\eta
\end{equation}
where $0< \alpha < \beta $ are positive constants, $1<p<N$, and $b$ and $\fff$ are positive functions verifying
$b\in L^{N,\infty}(\Omega)$ and $\fff\in L^{p}(\Omega)$. In view of Sobolev embedding theorem in Lorentz spaces \cite{O,A,GM}, by \eqref{limitatezza} and the assumptions on $b$ and $\fff$, for each $u\in\SobpO$ we have
\[A(x,u,\nabla u)\in L^{p'}(\Omega,\mathbb R^N)\]
Hence, we can define the quasilinear elliptic distributional  operator
\begin{equation}\label{loperatore}
-\divergenza A(x,u,\nabla u)
\end{equation}
setting for any $w\in \SobpO$
\begin{equation}
\langle -\divergenza A(x,u,\nabla u),w\rangle =\int_\Omega \langle A(x,u,\nabla u),\nabla w\rangle \dx x\,.
\end{equation}
Given $\Phi\in\SobdualO$, we study the Dirichlet problem
\begin{equation}\label{Dirichlet}
\left\{\begin{array}{c}
           -\divergenza A(x,u,\nabla u)= \Phi\vspace{7pt}\\
          u\in \SobpO
         \end{array}
\right.
\end{equation}
By a solution to Problem \eqref{Dirichlet} we mean a function $u \in W_0^{1,p}(\Omega)$ such that
\begin{equation}\label{soluzint}
 \int_\Omega   A(x,u,\nabla u)\nabla w \dx x
= 
 \langle 
\Phi , w 
 \rangle\,,
 \qquad \forall w \in C^\infty_0(\Omega)\,,
\end{equation}
where  $\langle\cdot,\cdot \rangle$ denotes the duality product of 
$  W ^{-1,p^\prime}(\Omega)$
and
$  W_0^{1,p}(\Omega)$. Clearly, \eqref{soluzint} extends to all $w \in W_0^{1,p}(\Omega)$.\par
Our structural conditions allow for a singular lower order term depending on $u$. As an example, we consider the following operator

\begin{equation}\label{modello}
A(x,u,\xi):=\langle \mathcal H(x)\xi , \xi \rangle ^{\frac {p-2} 2} \mathcal H(x) \xi + B(x) |u|^{p-2}u
\end{equation}
with $1<p<N$. Here $\mathcal H=\mathcal H(x)\colon \Omega \rightarrow \mathbb R^{N \times N}$ is a symmetric, bounded  matrix field of  $\mathbb R^{N \times N}$ such that
\[
\langle \mathcal H(x)\xi , \xi \rangle \geq \alpha |\xi|^2 \quad \text{for a.e. $x \in \Omega$ and for all }\xi \in \mathbb R^N\,,
\]
for a given $\alpha>0$.
The vector field  $B\colon \Omega \rightarrow \mathbb R^{N}$ is a measurable function satisfying $|B(x)|\leqslant (b(x))^{p-1}$ for a.e.\ $x \in \Omega$ and for some $b \in L^{N, \infty }(\Omega)$.

\medskip

The feature of Problem \eqref{Dirichlet} is the lack of coercivity for the operator \eqref{loperatore}
and the singularity in the lower order term due to property of $b$. 
It is well known that, if the operator in (1.4)-(1.5) is coercive, then a solution to problem (1.6) exists. It can for instance be shown  by monotone operator theory \cite{LL,HB B,Br,BBM}. 

On the other hand, 
the existence of a bounded solution 
can be expected when $\Phi$ and $b$ are sufficiently smooth. For example, in the model case and 
even for the corresponding linear case, 
a solution to 
Problem \eqref{Dirichlet} is bounded whenever $\Phi$ and $b$ are in $W^{-1,p^\prime }(\Omega)$ and 
$L^{p^\prime }(\Omega,\mathbb R^N)$, respectively, with $p^\prime > N/(p-1)$ (see \cite{S,GMZ}).\par
The space $L^{N,\infty}(\Omega)$ is slightly larger than
$L^{N}(\Omega)$. Nevertheless, there are essential differences between the case $b \in L^{N}(\Omega)$ 
(\cite{BoBumi,Droniou}), or even $b \in L^{N,q}(\Omega)$ (\cite{M1})
with $ N \leqslant q < \infty$, and the case $b\in L^{N,\infty}(\Omega)$. 
In $  L^{N,\infty}(\Omega)$ bounded functions are not dense. Furthermore, in $  L^{N,\infty}(\Omega)$ 
the norm is not absolutely continuous,  namely a function can have large norm even if restricted to a set  with   small measure. 

\medskip

Our  first  result is the following
\begin{theorem}\label{main}
Let $\Phi\in\SobdualO$. Under the assumptions
\eqref{coercivita},
\eqref{limitatezza} and
\eqref{monotonia},
if
\begin{equation}\label{lndeboleinfty} 
\dist_{L^{N,\infty}} (b,L^\infty (\Omega)) < \frac {\alpha^\frac1p}{S_{N,p}}
\end{equation}
then Problem~\eqref{Dirichlet} admits  a solution.  
\end{theorem}
Here  $S_{N,p}$ denotes the Sobolev constant of Theorem 
\ref{lorentz} below. As an illustration, we state the following  immediate consequence. We denote by $\chiusura N(\Omega)$ the closure of $L^\infty(\Omega)$ in $L^{N,\infty}(\Omega)$.
\begin{corollary}\label{corollary}
Assume \eqref{coercivita},
\eqref{limitatezza} and
\eqref{monotonia}, with $b\in \chiusura N(\Omega)$.
Then Problem~\eqref{Dirichlet} admits  a solution, for every $\Phi\in\SobdualO$.  
\end{corollary}
The closure $\chiusura N(\Omega)$ contains for example all Lorentz spaces $L^{N,q}(\Omega)$, for $1<q<+\infty$, see Subsection~\ref{funct sp}.
\par
It is not clear if the bound in \eqref{lndeboleinfty} is sharp. 
However, existence of a solution could fail if 
no restriction on  the distance is imposed, even in the liner case, 
as the Example
\ref{Ex1}
in Subsection \ref{example sec} shows
(see also \cite{GMZ3}).
Notice that condition \eqref{lndeboleinfty} does not imply the smallness of the norm of $b$ in 
$L^{N,\infty}(\Omega)$ (see \cite{GMZ3}).

In the case $p=2$ existence results analogous to that of Theorem  \ref{main} have been proved in 
\cite{GGM,GMZ3,Z} and in \cite{DiG Z,RZ,Z2} when the principal part has a coefficient  bound in BMO (i.e. the space of functions of bounded mean oscillation). 
We explicitly remark that in this 
context the operator \eqref{loperatore} has the same integrability properties of the principal part (see also~\cite{KK}). 
The evolutionary counterpart of Problem~\eqref{Dirichlet}
has been studied in~\cite{FarMos}. 
Other related results can be found in 
\cite{AlfMo,
KTAIPENG,MV}.

\medskip

An additional difficulty in proving Theorem~\ref{main} lies in the lack of compactness that  the operator
\begin{equation}\label{lack}
u\in \SobpO\quad \mapsto \quad \big(b(x)\,|u|\big)^{p-1}\in L^{p'}(\Omega)
\end{equation}
exhibits in the case $b\in L^{N,\infty}(\Omega)$,
in contrast with the case $b\in L^{N}(\Omega)$
(see Example 
\ref{Ex2}
 in Subsection \ref{example sec}).
In order to overcome this issue, first we consider the case in which $b\in L^\infty(\Omega)$. Under this assumption, we deduce the existence of a solution to  Problem~\eqref{Dirichlet} by means of Leray--Schauder fixed point theorem. The a priori estimate required follows from a lemma that could be interesting in itself (see 
Lemma \ref{lemma weak compact}
below).

\medskip

In order to  reduce the general case $b\in L^{N,\infty}(\Omega)$  to
the previous one, we consider a sequence of approximating problems, defined essentially  by truncating the vector field $A=A(x,u,\xi)$ in the $u$--variable.
A bound on the sequence of the solutions is achieved due to the assumption
\eqref{lndeboleinfty}. 

We emphasize
 that our result is also new when $b\in L^{N}(\Omega)$, in the sense that our approach allows us to treat the general nonlinear operator in 
\eqref{Dirichlet}.

Finally,  by testing the problems with a suitable admissible test functions, we show that the sequence of solutions to the approximating problems is compact and its limit is a solution to the original problem \eqref{Dirichlet}.

\medskip

In Section~\ref{pb ostacolo}, we show that our approach is robust enough to handle also the corresponding obstacle problem.   
We prove an existence result in the same spirit of \cite{GMZ2} (where the case $p=2$ is taken into account).

\medskip
In Section~\ref{reg sol} we present a regularity result.
When 
$b \in L^N(\Omega)$, 
the study of the higher integrability of a solution to \eqref{Dirichlet} 
has been developed in \cite{GG,Giustibook} 
by using the theory of quasiminima. Local summability properties have been recently considered in 
\cite{CDLP,KangKim}
in the linear case. 
Here, following
\cite{GMZ},
 we prove  higher  summability of a solution  $u$ to \eqref{Dirichlet} according to that of the  data $\Phi$ and $\fff$.

 \section{Preliminaries and examples}\label{section2}
\subsection{Notation and function spaces}\label{funct sp}
Let $\Om$ be a bounded domain in $\mathbb R^N$.
Given $1<p<\infty$ and $1\leqslant q<\infty$, the Lorentz space $L^{p,q}(\Om)$ consists of all measurable functions $f$ defined on $\Omega$ for which the quantity
\begin{equation}\label{norma Lorentz}
\|f\|_{p,q}^q=p\int_{0}^{+\infty} |\Om_t|^{\frac{q}{p}}t^{q-1}\dx t
\end{equation}
is finite, where $\Om_t= \left\{ x\in \Om: |f(x)|>t \right\}$ and $|\Om_t|$ is the Lebesgue measure of $\Om_t$, that is, $\lambda_f(t)=|\Om_t|$ is the distribution function of $f$. Note that $\|\cdot\|_{p,q}$  is equivalent to a norm and $L^{p,q}$ becomes a Banach space when endowed with it (see \cite{O,bs,g}). For $p=q$, the Lorentz space
$L^{p,p}(\Omega)$ reduces to the Lebesgue space $L^p(\Omega)$. For $q=\infty$, the class $L^{p,\infty}(\Omega)$ consists of all measurable functions $f$ defined on $\Omega$ such that
\begin{equation*}\label{2.1}
\|f\|^p_{p,\infty}=\sup_{t>0}t^{ p}|\Om_t|<+\infty
\end{equation*}
and it coincides with the Marcinkiewicz class, weak-$L^p(\Omega)$.

For Lorentz spaces the following inclusions hold
\[
L^r (\Om)\subset L^{p,q}(\Om)\subset  L^{p,r} (\Om) \subset L^{p,\infty}(\Om)\subset L^q(\Om),
\]
whenever $1\leqslant q<p<r\leqslant \infty.$ Moreover, for $1<p<\infty$, $1\leqslant q\leqslant \infty$ and $\frac 1 p+\frac 1 {p'}=1$, $\frac 1 q+\frac 1 {q'}=1$, if $f\in L^{p,q}(\Omega)$, $g\in L^{p',q'}(\Omega)$ we have the H\"{o}lder--type inequality
\begin{equation*}\label{holder}
\int_{\Omega}|f(x)g(x)|\dx x \leqslant \|f\|_{p,q}\|g\|_{p',q'}.
\end{equation*}

As it is well known, $L^\infty(\Om)$ is not dense in $L^{p,\infty}(\Om)$. 
For a function $f \in 
L^{p,\infty} (\Om)$ we define
\begin{equation}\label{dist}
\dist_{L^{p,\infty} (\Om) } (f,L^\infty (\Om))  =\inf_{g\in L^\infty(\Om)} \|f-g\|_{L^{p,\infty}(\Om)}.
\end{equation}
In order to characterize the distance in \eqref{dist}, 
we introduce for all  $k>0$ the truncation operator
\begin{equation*}\label{201206272}
T_k(y):=\frac{y}{|y|}\min\{|y|, k\}\,.
\end{equation*}
It is easy to verify that
\begin{equation}\label{distlim}
\lim_{k\to\infty}\|f-T_kf\|_{p,\infty} 
=
\dist_{L^{p,\infty} (\Om) } (f,L^\infty (\Om))\,.
\end{equation}
We denote by $\chiusura p(\Omega)$ the closure of $L^\infty(\Omega)$. We have (see \cite[Lemma~2.3]{GZ})
\begin{equation}\label{caratterizzazione chiusura}
f\in\chiusura p(\Omega)\iff \lim_{t\to+\infty}t\,[\lambda_f(t)]^{1/p}=0\,.
\end{equation}
Clearly, for $1\leqslant q<\infty$ we have $L^{p,q}(\Omega)\subset \chiusura p(\Omega)$, that is, any function in $L^{p,q}(\Omega)$ has  vanishing distance zero to $L^\infty(\Omega)$. Indeed, $L^\infty(\Omega)$ is dense in $L^{p,q}(\Omega)$, the latter being continuously embedded into $L^{p,\infty}(\Omega)$. Actually, the inclusion also follows from \eqref{caratterizzazione chiusura}, since $\lambda_f(t)=|\Omega_t|$ is decreasing and hence the convergence of the integral at \eqref{norma Lorentz} implies the condition on the right of \eqref{caratterizzazione chiusura}.\par
Assuming the origin $0\in\Omega$, a typical  element of $L^{N,\infty}(\Om)$ is $b(x)=B/|x|$, with $B$ a positive constant. An elementary calculation shows that
\begin{equation}\label{esempio-dist}
\dist_{L^{N,\infty} (\Om) } (b,L^\infty (\Om))=B\,\omega_N^{1/N}
\end{equation}
where $\omega_N$ stands for the Lebesgue measure of the unit ball of $\Rn$.\par\medbreak
The Sobolev embedding theorem in Lorentz spaces reads as
\begin{theorem}[\cite{O,A,GM}]\label{lorentz}
Let us assume that $1<p<N$, $1\leqslant q\leqslant p$, then every function $g\in W_0^{1,1}(\Omega)$ verifying $|\nabla g|\in L^{p,q}(\Om)$ actually belongs to $L^{p^*,q}(\Om)$, where $p^*=\frac{Np}{N-p}$ and
$$
\|g\|_{p^*,q}\leqslant S_{N,p}\|\nabla g\|_{p,q}
$$
where $S_{N,p}$ is the Sobolev constant.
\end{theorem}

\subsection{A version of the Leray--Schauder fixed point theorem}
We shall use the well known Leray--Schauder fixed point theorem in the following form (see \cite[Theorem 11.3 pg. 280]{gt}). A continuous mapping between two Banach spaces is called compact if the images of bounded sets are precompact. 
\begin{theorem}\label{LerSch2}
Let $\FFF$ be a compact mapping of a Banach space $X$ into itself, and suppose there exists a constant $M$ such that $\|x\|_{X}<M$ for all $x\in X$ and $t\in [0,1]$ satisfying $x=t\FFF(x).$ Then, $\FFF$ has a fixed point. 
\end{theorem}

\subsection{Critical examples}\label{example sec}
Our first example shows that the only assumption that $b \in L^{N,\infty} (\Omega)$
does not guarantee the existence of a solution to Problem \eqref{Dirichlet}. 
\begin{example}\label{Ex1}
Let $\Omega$ be the unit ball. For $\frac N2<\gamma+1\leqslant N$, the problem
\begin{equation}\label{controesempio}
\left\{
\begin{array}{cl}
\ds -\Delta u-\divergenza\left(\gamma\,u\,\frac x{|x|^2}\right)=-\divergenza\left(\frac x{|x|^{N-\gamma}}\right)&\text{ in }\Omega\\
\ds u=0&\text{ on }\partial\Omega
\end{array}
\right.
\end{equation}
does not admit a solution.
Assume to the contrary that  $u$ is a solution of \eqref{controesempio}. 
In the right hand side of the equation we recognize that
\[\frac x{|x|^{N-\gamma}}=\nabla v(x)\,,\]
where $v\in W^{1,2}_0(\Omega)$
 is given by
\[v(x)=\left\{
\begin{array}{ll}
\ds \frac1{2-N+\gamma}\,(|x|^{2-N+\gamma}-1)&\text{ for }\gamma\not=N-2\\
\ds \vrule width 0pt height 1.2em
\log|x|&\text{ for }\gamma=N-2
\end{array}
\right.\]
Moreover, $v$ solves the adjoint problem
\[\left\{
\begin{array}{cl}
\ds -\Delta v+\gamma\,\frac x{|x|^2}\cdot \nabla v=0&\text{ in }\Omega\\
\ds v=0&\text{ on }\partial\Omega
\end{array}
\right.\]
Testing the equation in \eqref{controesempio} by $v$ we have
\[\int_\Omega |\nabla v|^2\dx x=0\,.\] 
which readly implies $v \equiv 0$ in $\Omega$, which is clearly not the case.  \qed
\end{example}
Next example shows that for the complete operator
\[\divergenza A(x,u,\nabla u)+B(x,u,\nabla u)\]
in general we do not have existence, even in the linear case.
\begin{example}\label{operatore completo}
Let $\lambda$ be an eigenvalue of Laplace operator and $w$ a corresponding eigenfunction
\[\left\{
\begin{array}{ll}
\ds -\Delta w=\lambda\, w\\
\ds w\in\Sob2\Omega\setminus\{0\}
\end{array}
\right.\]
Then the equation
\[-\Delta u-\lambda\,u=w\]
has no solution of class $\Sob2\Omega$.
\end{example}
Our final example shows that  
compactness of the operator
\eqref{lack} in the Introduction could fail.
\begin{example}\label{Ex2}
Assume $N\geqslant 2$ and  $1<p<N$. Let $\Omega$ be the ball of $\mathbb R^N$ centered at the origin of radius $3$. 
Our aim is to construct a sequence of functions
$\{
u_n
\}_{n\in\mathbb N}$
in $W^{1,p}_0(\Omega)$ and a function $b\in L^{N,\infty}(\Omega)$
such that $\{
\nabla u_n
\}_{n\in\mathbb N}$ is bounded in $L^p(\Omega,\mathbb R^N)$, however it is not possible to extract from $\{(b|u_n|)^{p-1} \}_{n\in\mathbb N}$ any subsequence strongly converging  
$L^{p^\prime}(\Omega)$. 
To this aim, let $$ b(x):=\frac 1 {|x|} $$ and
\[
\gamma : = 1 - \frac N p\,.
\]
We define a sequence 
$\{
u_n
\}_{n\in\mathbb N}$
setting for $x \in \Omega$  
\begin{equation}\label{controesempio2}
\begin{split}
u_1(x) & :=
\left\{
\begin{array}{cl}
 1-2^\gamma &\text{if $|x|<1$}\vspace{7pt}\\
 |x|^\gamma -2^\gamma &\text{if $1 \leqslant |x| < 2$}\vspace{7pt}\\
 0 &\text{if $|x|\geqslant 2$}
\end{array}
\right.
\\
u_n (x)&  : = n^{-\gamma} u_1 (nx) \qquad \text{for $n\geqslant 2$}  
\end{split}
\end{equation}
Observe that 
$u_n  \in W^{1,p}_0(\Omega)$ since 
\begin{equation}\label{contr2.1}
\begin{split}
 | \nabla u_n(x)  | =
\left\{
\begin{array}{cl}
 |\gamma |  
 |x|^\gamma  &\text{if $\frac 1 n \leqslant |x| < \frac 2 n$}\vspace{7pt}\\
 0 &\text{otherwise}
\end{array}
\right.
\end{split}
\end{equation}
and 
\begin{equation}
\| \nabla u_n\|^p_{L^p(\Omega)} = |\gamma|^p N \omega_N \log 2\,,
\end{equation}
where $\omega_N$ denotes the measure of the unit ball of $\mathbb R^N$.
In particular, $\| \nabla u_n\|^p_{L^p(\Omega)}$ is independent of $n$.
On the other hand, a direct calculation shows that
\begin{equation}
\begin{split}
\left\|
\left(
b|u_n|
\right)^{p-1}
\right\|^{p^\prime} _{L^{p^\prime} (\Omega)}
& =
\int_{|x|<\frac 3 n} (b |u_n|)^p \dx x
\\
&
=\int_{|x|<\frac 1 n} (b |u_n|)^p \, dx + \int_{\frac 1 n \leqslant |x|<\frac 2 n} (b |u_n|)^p \dx x
\\
&
= 
N\omega_N
\left[
\frac{ (1-2^\gamma)^p}{N-p} + \int_1^2    r^{N-p} \left(r^\gamma-2^\gamma\right)^p    \, \frac{\mathrm d r} r
\right]
\end{split}
\end{equation}
Hence, we see that the  norm of $(b|u_n|)^{p-1}$ in $L^{p^\prime}(\Omega)$
is independent of $n$ as well 
and strictly positive.
On the other hand, $(b  |u_n|)^{p-1}\rightarrow 0$ pointwise  in $\Omega$ 
and this readily implies that 
there is no subsequence of 
$\{ (b|u_n|)^{p-1} \}_{n\in \mathbb N}$ strongly converging in $L^{p^\prime}(\Omega)$.
\qed
\end{example}
\subsection{An elementary lemma}
\begin{lemma}\label{lemma elementare}
Assume $f_n\to f$ a.e. Moreover, let $g_n$, $n\in \naturale$, and $g$ in $L^q$, $1\leqslant q<+\infty$, verify $g_n\to g$ a.e., 
$|f_n|\leqslant g_n$ a.e., $\forall n\in\naturale$, and
\[\int_\Omega g_n^q\dx x\to \int_\Omega g^q\dx x\,.\]
Then $f_n,f\in L^q$ and
\[f_n\to f\text{ in }L^q\,.\]
\end{lemma}
It suffices to apply Fatou lemma to the sequence of nonnegative functions
\[F_n=2^{q-1}(g_n^q+g^q)-|f_n-f|^q\,.\]
\section{A weak  compactness result}
The aim of this section is to establish a weak compactness criterion in the space $W^{1,p}_0(\Omega)$
that has an interest by itself. 
\begin{lemma}\label{lemma weak compact}
Let $\mathcal B$ be a nonempty subset of $W^{1,p}_0(\Omega)$. Assume that there exists a constant $C>0$ such that
\begin{equation}\label{stima vkn3 bis}
\|\nabla   u \|_{L^p(\Omega\setminus \Omega_\sigma)}^p\leqslant 
C\left(
1+\|  u \|_{L^p(\Omega\setminus \Omega_\sigma)}^p
\right)
\end{equation}
for any $\sigma>0$ and $u \in \mathcal B$, where $\Omega_\sigma:=\{x\in \Omega\colon\,|u(x)| \geqslant  \sigma\}$. Then, there exists a constant $M>0$ such that 
\begin{equation}\label{3.23quater}
\|  u \|_{W^{1,p}(\Omega)}\leqslant M\,
\end{equation}
for any 
$u \in \mathcal B$.
\end{lemma}
\begin{proof}
We argue by contradiction and assume $\mathcal B$ unbounded. Then we construct a sequence 
$\{u_k\}_k$ in $\mathcal B$  such that
\[\| u_k \|:=\| \nabla u_k \|_{p}\rightarrow \infty\]
as $k\rightarrow \infty$.
By \eqref{stima vkn3 bis} we get, for any $k \in \mathbb N$ and $\varepsilon >0$
\begin{equation}\label{3.3ter}
\int_\Omega |\nabla T_{\ee\|u_k\|}u_k|^p
\dx x
\leqslant 
C\left(
1+
\int_{\Omega}  |u_k|^p\chi_{\{|u_k|<\ee\|u_k\|\}} \dx x
\right)
\end{equation}
We set 
\[v_k=\frac{u_k}{\|u_k\|}\,.\]
Hence, there exists $v \in W^{1,p}_0(\Omega)$ such that (up to a subsequence)
$v_k\rightharpoonup v$ weakly in $\Sob p{}$, $v_k\to v$ strongly in $L^p$ and $v_k(x)\to v(x)$ for a.e.~$x\in \Omega$.
Notice that
\[\frac {T_{\varepsilon \|u_k\|  } u_k } {\|u_k\|}= T_{\ee} v_k\,,\]
thus $\nabla {T_{\varepsilon \|u_k\|  } u_k }=0$ on the set $\{x\in\Omega:|v_k(x)|\geqslant \ee\}$. 
Dividing 
\eqref{3.3ter}
by $\|u_k\|^{p}$  we have
\begin{equation}\label{stima vk ter}
  \int_\Omega |\nabla T_{\ee}v_k|^p \dx x \leqslant
C\left(
\|u_k\|^{ -p} + 
\int_{\Omega}  |v_k|^p\chi_{\{|v_k|<\ee    \}} \dx x
\right)
\end{equation}
Now, we let $k\to+\infty$. To this end, we note that $T_\ee v_k\rightharpoonup T_\ee v$ weakly in $\Sob p{\Omega}$ and $T_\ee v_k\to T_\ee v$ strongly in $L^p(\Omega)$. In the left hand side of \eqref{stima vk ter}, we use semicontinuity of the norm with respect to weak convergence, while in the right hand side we observe that $\|u_k\|^{-1}\to0$. Moreover, if
\begin{equation}\label{e eccezionale ter}
|\{x\in\Omega:|v(x)|=\ee\}|=0\,,
\end{equation}
then we have $\chi_{\{|v_k|<\ee\}}\to \chi_{\{|v|<\ee\}}$ a.e.\ in $\Omega$ and hence
\[v_k\,\chi_{\{|v_k|<\ee\}}\to v\,\chi_{\{|v|<\ee\}}\]
strongly in $L^p$. Note that the set of values $\ee>0$ for which \eqref{e eccezionale ter} fails is at most countable.
Thus, we end up with the following estimate
\begin{equation}\label{stima v ter}
  \int_\Omega |\nabla T_{\ee}v|^p \dx x \leqslant C \int_{\Omega}|v|^p\chi_{\{|v|<\ee\}} \dx x
\end{equation}
Using Poincar\'e inequality in the left hand side, this yields
\[\ee^p\,|\{x:|v|\geqslant\ee\}|\leqslant C\,\ee^p\,|\{x:0<|v|<\ee\}|\,.\]
Passing to the limit as $\ee\downarrow 0$ (assuming \eqref{e eccezionale ter}), we deduce
\[|\{x:|v|>0\}|=0\,,\]
that is, $v(x)=0$ a.e. Once we know that $v_k\rightharpoonup 0$ weakly in $\SobpO$, the above argument (formally with $\ee=+\infty$, i.e.~without truncating $v_k$) actually shows that $v_k\to0$ strongly in $\SobpO$, compare with \eqref{stima v ter}, and this is not possible, as $\|v_k\|=1$, for all $k$. 
\end{proof} 

\section{Proof of Theorem \ref{main}}\label{Section3}

\subsection{The case of bounded coefficient}\label{coefficienti limitati}
In this subsection we assume  $b\in L^\infty(\Omega)$.  For a given function  $v\in L^p(\Omega)$, we define the vector field on $\Omega\times\Rn$
\begin{equation}\label{campoAv}
A_v(x,\xi):=A(x,v(x),\xi)
\end{equation}
which satisfies similar conditions as $A$, namely
\begin{equation}\label{coercivita-v}
\langle A_v(x,\xi),\xi\rangle\geqslant \alpha|\xi|^p-\big(b(x)\,|v|\big)^p-\fff(x)^{p}
\end{equation}
\begin{equation}\label{limitatezza-v}
|A_v(x,\xi)|\leqslant \beta|\xi|^{p-1}+\big(b(x)\,|v|\big)^{p-1}+\fff(x)^{p-1}
\end{equation}
\begin{equation}\label{monotonia-v}
\langle A_v(x,\xi)- A_v(x,\eta ), \xi -\eta \rangle>0\qquad \text{for }\xi\ne\eta
\end{equation}
Hence, we can consider a quasilinear elliptic 
operator similar to \eqref{loperatore}
\begin{equation}\label{loperatore-v}
u\in \SobpO\mapsto-\divergenza A_v(x,\nabla u)\in \SobdualO
\end{equation}
defined by the rule
\begin{equation}
\langle -\divergenza A_v(x,\nabla u),w\rangle =\int_\Omega \langle A(x,v,\nabla u),\nabla w\rangle \dx x
\end{equation}
for any $w\in \SobpO$. The operator at \eqref{loperatore-v} is invertible. Indeed,
\begin{proposition}\label{la proposizione}
For every $\Phi\in\SobdualO$, there exists a unique $u\in \SobpO$ such that
\begin{equation}\label{equazione-v}
-\divergenza A_v(x,\nabla u)=\Phi
\end{equation}
Moreover, the mapping
\begin{equation}\label{mappa continua}
(v,\Phi)\in L^p(\Omega)\times \SobdualO\mapsto u\in \SobpO
\end{equation}
is continuous.
\end{proposition}
\begin{proof}
Existence of a solution is classical, see e.g.~\cite{LL}, \cite[pg.~27]{Br}, or \cite[Th\'eor\`eme~2.8, pg.~183]{L}. Uniqueness trivially holds by monotonicity.\par
For the sake of completeness, we prove continuity of the map \eqref{mappa continua}.
Given $v_n\to v$ in $L^p(\Omega)$ and $\Phi_n\to\Phi$ in $\SobdualO$,
let $u_n\in \SobpO$ solve
\begin{equation}\label{202005251}
-\divergenza A(x,v_n,\nabla u_n)=\Phi_n\,.
\end{equation}
The sequence $\{u_n\}_n$ is clearly bounded, hence we may assume $u_n \rightharpoonup u$ weakly in $\SobpO$. Moreover, testing equation \eqref{202005251} with $u_n-u$, we have
\begin{equation}\label{202005022}
\lim_{n\rightarrow \infty}
\int_\Omega A(x,v_n,\nabla u_n) (\nabla u_n-\nabla u)\dx x=\lim_{n\rightarrow \infty}\langle \Phi_n,u_n-u\rangle=  0\,.
\end{equation} 
On the other hand, we easily see that $A(x,v_n,\nabla u)\to A(x,v,\nabla u)$ strongly in $L^{p'}(\Omega,\reale^N)$ and thus \eqref{202005022} implies
\begin{equation}\label{202005023}
\lim_{n\rightarrow \infty}
\int_\Omega \left[ A(x,v_n,\nabla u_n) -  A(x,v_n,\nabla u) \right]\nabla (u_n-u)\dx x= 0\,.
\end{equation}
The integrands in \eqref{202005023} are nonnegative by monotonicity. Hence, arguing as in the proof of \cite[Lemma~3.3]{LL}, we also get $\nabla u_n(x)\to\nabla u(x)$ a.e.\ in $\Omega$, and
\[A (x,v_n,\nabla u_n)\rightharpoonup A  (x,v ,\nabla u )\]
weakly in $L^{p'} (\Omega,\mathbb R^N)$. Combining this with \eqref{202005022} yields
\begin{equation}\label{202005033}
\lim_{n\rightarrow \infty}
\int_\Omega A(x,v_n,\nabla u_n) \nabla u_n\dx x=\int_\Omega A(x,v,\nabla u) \nabla u\dx x\,.
\end{equation}
By coercivity condition \eqref{coercivita}, we deduce
\[\alpha |\nabla u_n|^p\leqslant A(x,v_n,\nabla u_n) \nabla u_n+(b|v_n|)^p+\fff^{p}\]
Trivially $\int_\Omega(b|v_n|)^p\dx x$ converges to $\int_\Omega(b|v|)^p\dx x$. In view of \eqref{202005033}, by Lemma~\ref{lemma elementare} we get $u_n \to u$ strongly in $\SobpO$, and $u$ solves the equation
\[-\divergenza A(x,v,\nabla u)=\Phi\,.\]
\end{proof}
In view of Rellich Theorem, we have
\begin{corollary}\label{corollary compattezza}
For fixed $\Phi\in\SobdualO$, the mapping
\begin{equation}\label{risolvente}
\FFF\colon v\in \SobpO\mapsto u\in \SobpO
\end{equation}
which takes $v$ to the unique solution $u$ of equation~\eqref{equazione-v} is compact.
\end{corollary} 
Now we  state an existence result  to Problem \eqref{Dirichlet} 
when $b\in L^\infty(\Omega)$.  

\begin{proposition}\label{b-limitata}
Let 
\eqref{coercivita}, \eqref{limitatezza} and \eqref{monotonia} 
be in charge 
with $b\in L^\infty(\Omega)$. 
 Then Problem \eqref{Dirichlet}  has a solution $u \in W_0^{1,p} (\Omega)$. 

\end{proposition}
\begin{proof}
If $\mathcal{F}$ is the operator defined in Corollary \ref{corollary compattezza}, clearly
a fixed point of $\FFF$ is a solution to Problem~\eqref{Dirichlet}. To apply Leray-Schauder theorem, we need an a~priori estimate on the solution $u\in \SobpO$ of the equation
\[u=t\FFF[u]\]
that is
\begin{equation}\label{equazione-k ter}
-\divergenza A\left(x,u ,\frac1{t }\,\nabla u \right)=\Phi\,,
\end{equation}
as $t\in{}]0,1]$ varies.
By using $T_\sigma u$ with $\sigma >0$ as a test function in  \eqref{equazione-k ter} we get
\begin{equation}
\int _\Omega 
\left \langle
A\left(x,u ,\frac1{t }\,\nabla u \right)
,
\nabla T_\sigma u
\right \rangle\dx x
=
\left \langle
\Phi
,
T_\sigma u
\right \rangle
\end{equation}
Therefore, using the point-wise condition \eqref{coercivita-v} we get
\begin{equation}\label{stima vk0 ter}
\alpha\, t ^{1-p}\int_\Omega |\nabla T_{\sigma} u |^p \dx x
\leqslant \|\Phi\|\,\|\nabla T_\sigma u_k\|_p+\int_{\Omega}\Big[b(x)^p|u |^p\chi_{\{|u |< \sigma \} }+\fff(x)^{p}\Big]\dx x
\end{equation}
As $0<t\leqslant 1$, by Young inequality \eqref{stima vk0 ter} yields
\begin{equation}\label{final proposition ter}
\frac \alpha 2 
\int_\Omega |\nabla T_{\sigma} u |^p \dx x
\leqslant
 \|\Phi\|^{p^\prime}
+
\| b \|_\infty ^p 
\int_{\Omega}
|u|^p \chi_{ \left\{   |u|\leqslant \sigma     \right\}  }
\dx x
+
\|\varphi\|_p^{p}
\end{equation}
The conclusion follows by Lemma \ref{lemma weak compact}. 
\end{proof}
\subsection{The approximating problems}\label{approssimazione}
For each $n\in \naturale$, we set
\begin{equation}\label{theta}
\vartheta_n(x)=\frac{T_nb(x)}{b(x)},\qquad \mbox{ a.e. }x\in \Om\,,
\end{equation}
and define the vector field 
\begin{equation}\label{A n}
A_n\colon (x,u,\xi)\in \Omega\times \reale\times \Rn\mapsto \Rn
\end{equation}
letting
\begin{equation}\label{theta2}
A_n(x,u,\xi)=A(x,\vartheta_n u,\xi)\,
\end{equation}
The vector field $A_n$ has similar properties as $A$, with $b$ replaced by $T_nb$. More precisely, 
\begin{equation}\label{coercivita n}
\langle A_n(x,u,\xi),\xi\rangle\geqslant \alpha|\xi|^p-\big(T_nb(x)\,|u|\big)^p-\fff(x)^{p}
\end{equation}
\begin{equation}\label{limitatezza n}
|A_n(x,u,\xi)|\leqslant \beta|\xi|^{p-1}+\big(T_nb(x)\,|u|\big)^{p-1}+\fff(x)^{p-1}
\end{equation}
\begin{equation}\label{monotonia n}
\langle A_n(x,u,\xi)- A_n(x,u,\eta ), \xi -\eta \rangle>0\qquad \text{for }\xi\ne\eta
\end{equation}
Applying Proposition~\ref{b-limitata} with $A_n$ in place of $A$, fixed $\Phi\in \Sobdual\Omega$, we find $u_n\in \SobpO$ such that
\begin{equation}\label{equazione n}
-\divergenza A_n(x,u_n,\nabla u_n)= \Phi\,.
\end{equation}
Notice that we have, for $\sigma>0$
\begin{equation}\label{stima vkn}
\alpha\, \int_\Omega |\nabla T_\sigma u_n|^p
\dx x
\leqslant \|\Phi\|\,\|\nabla T_\sigma u_n\|_p+\int_{\Omega}\Big[ (T_n b) ^p\,|u_n|^p
\chi_{\{|u_n|<\sigma\}}
+\fff^{p}\Big]\dx x
\end{equation}
which implies
\begin{equation}\label{stima vkn norme}
\alpha^{\frac1p}\|\nabla T_\sigma u_n\|_p
\leqslant (\|\Phi\|\,\|\nabla T_\sigma u_n\|_p)^{\frac1p}+\|(T_n b)\,u_n
\chi_{\{|u_n|<\sigma\}}\|_p
+\|\fff\|_p
\end{equation}
Our next step consists in showing that the sequence $\{u_n\}_n$ 
is bounded in $\SobpO$. Let  $m$ be a positive integer  to be chosen later. For  $n\geqslant m$ we have
\[T_n b\leqslant T_mb+(b-T_mb)\] 
and hence 
\begin{equation}\label{migliore costante}
\|(T_n b)\,u_n \chi_{\{|u_n|<\sigma\}}\|_p\leqslant \|(T_m b)\,u_n \chi_{\{|u_n|<\sigma\}}\|_p+\|(b-T_m b)\,u_n \chi_{\{|u_n|<\sigma\}}\|_p
\end{equation}
Using H\"older and Sobolev inequalities we get
\[\|(b-T_mb)\,u_n
\chi_{\{|u_n|<\sigma\}}
\|_p\leqslant \|b-T_mb\|_{N,\infty}\|T_\sigma u_n\|_{p^*,p}\le
\,S_{N,p}\,
 \|b-T_mb\|_{N,\infty}\|\nabla T_\sigma u_n\|_p\]
Then \eqref{stima vkn norme} and \eqref{migliore costante} give
\begin{equation}\label{stima vkn3bis}
\begin{split}
\alpha^{\frac1p}\|\nabla T_\sigma u_n\|_p\leqslant 
(\|\Phi\|\,\|\nabla T_\sigma u_n\|_p)^{\frac1p}&+\|(T_m b)\,u_n \chi_{\{|u_n|<\sigma\}}\|_p
+\|\fff\|_p
\\
&
+S_{N,p}  \|b-T_m b\|_{L^{N,\infty}  (\Omega)  }  \|
\nabla T_\sigma u_n\|_{L^{p}  (\Omega) }
\end{split}
\end{equation}
By our assumption \eqref{lndeboleinfty}, the level $m$ can be chosen large enough so that 
\[
S_{N,p}  \|b-T_m b\|_{L^{N,\infty}  (\Omega)  }<\alpha^{\frac1p}  
\]
Then, by absorbing 
in 
\eqref{stima vkn3bis}  the latter term of the right hand side in the left hand side, we get
\begin{equation}\label{stima vkn3}
C\int_\Omega |\nabla T_\sigma u_n|^p\dx x\leqslant \|\Phi\|\,\|\nabla T_\sigma u_n\|_p+\int_{\Omega}\Big[(T_m b)^p\,|T_\sigma u_n|^p
\chi_{\{|u_n|<\sigma\}}
+\fff^{p}\Big]\dx x
\end{equation}
for a positive constant $C$ which is independent of $n$. Now, it is clear that \eqref{stima vkn3},  via Young inequality, allows us to apply Lemma~\ref{lemma weak compact}, then 
\begin{equation}\label{3.23bis}
\|u_n\|\leqslant M\,
\end{equation}
for a constant $M$ independent of $n$. 

\medskip

In the model case \eqref{modello}, it is easy to show that the operator $\FFF$ defined in \eqref{risolvente} is compact, also for $b\in L^N(\Omega)$ (see Remark \ref{202005252}  below). In the general case, in which
$b\in L^{N,\infty}(\Omega)$ we need more work. 

\subsection{Passing to the limit}\label{pass lim}
Now, we are in a position to conclude the proof of Theorem~\ref{main}.
Taking into account estimate \eqref{3.23bis}
we may assume
\begin{equation}\label{22}
\begin{split}
u_n \rightharpoonup u  & \qquad \text{in $\SobpO$ weakly} \\
u_n \rightarrow  u & \qquad \text{in $L^q(\Omega)$ strongly for any $q<p^*$, and also a.e.\ in $\Omega$}
\end{split}
\end{equation}
for some $u \in W^{1,p}_0(\Omega)$.
We shall conclude our proof showing that $u$ solves Problem~\eqref{Dirichlet}.
In the rest of our argument, we let for simplicity   $\gamma(t):=\arctan t$. Obviously, $\gamma \in C^1(\mathbb R)$, $|\gamma (t)|\leqslant |t|$ and $0 \leqslant \gamma^\prime(t) \leqslant 1$ for all $t \in \mathbb R$. In particular,
$\gamma$ is Lipschitz continuous in the whole of $\mathbb R$ and therefore 
\[
u_n,u\in \SobpO \quad \Longrightarrow \quad \gamma(u_n-u)\in \SobpO\,.
\]
Moreover, since $\gamma(0)=0$ we have 
\begin{equation}\label{23}
\gamma(u_n-u) \rightharpoonup 0    \qquad \text{in $\SobpO$ weakly}\,.  
\end{equation}
Testing equation \eqref{equazione n} with the function $\gamma(u_n-u)$ we get
\[
\int_\Omega A_n(x,u_n,\nabla u_n) \nabla \gamma(u_n-u)\dx x= \left\langle \Phi , \gamma(u_n-u) \right \rangle
\]
where $\nabla \gamma(u_n-u) = \gamma^\prime(u_n-u) (\nabla u_n - \nabla u)$.
In view of \eqref{23} we necessarily have
\begin{equation}\label{24}
\lim_{n\rightarrow \infty}
\int_\Omega A_n(x,u_n,\nabla u_n) \nabla \gamma(u_n-u)\dx x=  0\,.
\end{equation} 

We claim that
\begin{equation}\label{26}
\lim_{n\rightarrow \infty} \int_\Omega A_n(x,u_n,\nabla u) \nabla \gamma(u_n-u)\dx x= 0\,.
\end{equation} 
In order to prove \eqref{26}, since $\nabla u_n - \nabla u\rightharpoonup 0$, it suffices to show that
\begin{equation}\label{compattezza}
\gamma'(u_n-u)\,A_n(x,u_n,\nabla u)= \frac { A_n(x,u_n,\nabla u)  }{1+|u_n-u|^2}\qquad \text{is compact in }L^{p'}\,.
\end{equation}
Preliminarily, we observe that  combining  \eqref{22} with the property that $\vartheta_n\rightarrow 1$ as $n\rightarrow \infty$, we have
\[
\frac { A_n(x,u_n,\nabla u)  }{1+|u_n-u|^2} \rightarrow  
A (x,u ,\nabla u) \qquad \text{a.e.\ in }\Omega\,.
\]
We are going to use Lemma~\ref{lemma elementare}. To this end, by \eqref{limitatezza n} we deduce that
\[
\left|\frac { A_n(x,u_n,\nabla u)  }{1+|u_n-u|^2}\right|^{p'}\leqslant C
\left[|\nabla u|^p+\varphi^{p} + (b|u|)^{p}+\frac {(b |u_n-u|)^{p} }{1+|u_n-u|^2}\right]
\]
for a positive constant $C=C(p,\beta)$. Hence, we can pass to the limit if $1<p\leqslant 2$. For $p>2$ we choose $s$ satisfying
\[
\frac {p^\ast}{p} < s < \frac {p^\ast}{p-2}\,,
\]
so that $ps^\prime<N$, and we conclude also in this case, further estimating with the aid of Young inequality
\[
\frac {(b |u_n-u|)^{p} }{1+|u_n-u|^2}\leqslant b^{ps'}+   |u_n-u|^{(p-2)s}\,.
\]
\medskip
Now, from \eqref{24} and \eqref{26} we get
\begin{equation}\label{27}
\lim_{n\rightarrow \infty}
\int_\Omega \left[ A_n(x,u_n,\nabla u_n) -  A_n(x,u_n,\nabla u) \right]\nabla \gamma(u_n-u)\dx x= 0\,.
\end{equation}
As the integrand is nonnegative, we have (up to a subsequence)
\[\left[ A_n(x,u_n,\nabla u_n) -  A_n(x,u_n,\nabla u) \right]\nabla \gamma(u_n-u)\to 0\]
a.e.\ in $\Omega$. Moreover, since $\gamma'(u_n-u)\to1$ a.e.\ in $\Omega$, the above in turn implies
\begin{equation}\label{202005021}
\left[ A_n(x,u_n,\nabla u_n) -  A_n(x,u_n,\nabla u) \right]\,(\nabla u_n-\nabla u)\to 0
\end{equation}
Arguing as in the proof of \cite[Lemma~3.3]{LL}, we see that
\begin{equation}\label{29.4}
\nabla u_n  \rightarrow  \nabla u \qquad \text{a.e.\ in $\Omega$}
\end{equation}
and
\begin{equation}\label{29.6}
A_n (x,u_n,\nabla u_n)
  \rightharpoonup
A  (x,u ,\nabla u )
 \qquad \text{in 
$L^{p^\prime} (\Omega,\mathbb R^N)$ weakly}
\end{equation}
and we conclude that $u$ is a solution to the original problem~\eqref{Dirichlet}.

\begin{remark}\label{202005252}
We discuss briefly the particular case in which the operator has the form
\[A(x,v,\xi)=A'(x,\xi)+A''(x,v)\,,\]
with
\[|A''(x,v)|\leqslant (b(x)\,|v|)^{p-1}+\fff(x)^{p-1}\,.\]
and $b\in L^N(\Omega)$ (see also \cite{BoBumi}). 
We can easily show that the operator $\FFF$ defined in \eqref{risolvente} is compact, also for $b\in L^N(\Omega)$. Indeed, equation \eqref{equazione-v} in this case becomes
\begin{equation}\label{equazione-v particolare}
-\divergenza A'(x,\nabla u)=\Phi+\divergenza A''(x,v)\,.
\end{equation}
Defined $\vartheta_n$ as in \eqref{theta}, each mapping
\[v\in \SobpO\mapsto A''(x,\vartheta_n\,v)\in L^{p'}(\Omega,\reale^N)\]
is clearly compact. Moreover,
\begin{equation}\label{202005253}
|A''(x,v)-A''(x,\vartheta_n\,v)|\leqslant 2[(b\,|v|)^{p-1}+\fff^{p-1}]\,\chi_{E_n}\,,
\end{equation}
where
\[E_n=\{x\in \Omega:|b(x)|>n\}\,.\]
Therefore, as $n\to+\infty$ we have
\[A''(x,\vartheta_n\,v)\to A''(x,v)\qquad \text{strongly in }L^{p'}(\Omega,\reale^N)\,,\]
the convergence being uniform when $v$ varies in a bounded subset of $\SobpO$, and compactness is preserved for the limit mapping.\par
An a~priori bound for solutions of equation
\[u=t\,\FFF[u]\]
can be easily obtained as above, splitting $b\in L^N(\Omega)$ as
\[b=T_m b+(b-T_mb)\]
for a sufficiently large $m$. Therefore, in this particular case the existence result of Theorem~\ref{main} follows simply applying Leray--Schauder fixed point theorem.\par
\end{remark}

\section{The obstacle problem}\label{pb ostacolo}
This section is devoted to the obstacle problem naturally related with problem   
\eqref{Dirichlet} (see \cite{KinStam} for a comprehensive treatment of the topic).
We again assume that 
\eqref{coercivita}, \eqref{limitatezza} and \eqref{monotonia} 
are in charge 
and we let $\Phi \in W^{-1,p}(\Omega)$.
Given a measurable function $\psi\colon \Omega  \rightarrow \overline {\mathbb  R}$,
where $\overline {\mathbb  R}:=[-\infty,\infty]$, we consider the convex subset 
of $\mathcal K_\psi(\Omega)$ of 
$
W_0^{1,p}(\Omega)
$ 
given by
\begin{equation}
\mathcal K_\psi (\Omega):=
\left\{
w \in W_0^{1,p}(\Omega) \colon
\,
w \geqslant \psi \text{ a.e.\ in }\Omega
\right\}.
\end{equation}
We will assume that $\mathcal K_\psi(\Omega)$ is nonempty. 
An element $u \in \mathcal K_\psi (\Omega)$ is a solution to the obstacle problem associated with \eqref{Dirichlet} if  the following  variational inequality holds
\begin{equation}\label{obst}
\begin{split}
\int_{\Omega }
\langle 
 A(x,u,\nabla u) ,   \nabla (w-u)  \rangle
\dx x
\geqslant
\langle \Phi,
 w-u
\rangle\,,
\qquad \forall w \in \mathcal K_\psi(\Omega)\,.
\end{split}
\end{equation}
As
$\mathcal K_\psi(\Omega)
\neq \emptyset
$,
we may assume without loss of generality that 
\begin{equation}\label{nonneg}
\text{$\psi\leqslant 0$ a.e. in $\Omega$.}
\end{equation}
In fact, if $g \in \mathcal K  _\psi(\Omega)$, then one can consider the operator defined by the vector field
\[
\tilde A(x,u,\xi):=A(x,u+g(x),\xi+\nabla g(x))\,,
\]
satisfying conditions similar to 
\eqref{coercivita},
\eqref{limitatezza} and
\eqref{monotonia}.
Now it is clear that, if  function
$\tilde u \in {\mathcal  K}  _{\psi-g}(\Omega)$ satisfies 
the following 
variational inequality
\begin{equation}\label{obst tilde}
\begin{split}
\int_{\Omega }
\langle \tilde A(x,\tilde u,\nabla \tilde u)  , \nabla (w-\tilde u)  \rangle
\dx x
\geqslant 
\langle \Phi,
 w-\tilde u
\rangle
\qquad \forall w \in  {\mathcal  K}  _{\psi-g}(\Omega)
\end{split}
\end{equation}
correspondingly 
$u = \tilde u + g$ is a solution to \eqref{obst}.  Notice that the obstacle function for problem \eqref{obst tilde} 
is nonpositive, as we are assuming for the original problem.

\begin{theorem}\label{ob theorem}
Let $\Phi\in\SobdualO$
and $\psi:\Omega\rightarrow [-\infty,0]$ be a measurable function. Under the assumption
\eqref{coercivita},
\eqref{limitatezza} and
\eqref{monotonia},
if \eqref{lndeboleinfty} holds,
then the ostacle problem~\eqref{obst} admits  a solution.  
\end{theorem}
\begin{proof}
We follow closely the arguments of 
Section \ref{Section3}. 
For each $n\in \naturale$, we consider the function $\vartheta_n$ as in \eqref{theta}
and define the vector fields
$
A_n=A_n (x,u,\xi) 
$
as in \eqref{theta2}.
We consider a sequence of obstacle problems provided by
\begin{equation}\label{obst n}
\begin{split}
\int_{\Omega }
\langle  A_n(x,u_n,\nabla u_n)  ,  \nabla (w-u_n)   \rangle
\dx x
\geqslant 
\langle \Phi,
w-u_n
\rangle\,,
\qquad \forall w \in \mathcal K_\psi(\Omega)\,.
\end{split}
\end{equation}
The existence of a 
solution $u_n \in \mathcal K_\psi(\Omega)$ to \eqref{obst n} is proven applying \cite[Th\'eor\`eme~8.2, pg.~247]{L} to the operator
\[-\divergenza A_n(x,v,\nabla u)\,,\]
for a fixed $v\in \SobpO$, and then using Leray--Schauder Theorem, arguing as in Subsection~\ref{coefficienti limitati}. Due to \eqref{nonneg},
for every
$ k  >0$
 the function
\[
w:=u_n - T_k u_n \in \mathcal K_\psi(\Omega) 
\]
is a test function for 
\eqref{obst n}.
Arguing 
as in Section \ref{approssimazione} we obtain 
\[
\|
u_n
\|\leqslant M
\]
with $M$ independent of $n$ (as in \eqref{3.23bis}). Therefore 
\eqref{22} holds
for some $u \in W^{1,p}_0(\Omega)$. It is clear from \eqref{22} itself that
\begin{equation}\label{u in K}
 u \in \mathcal K_\psi(\Omega) 
\end{equation}
As for Theorem \ref{main}, we shall  prove that $u$ is a solution to the original problem \eqref{obst}. We proceed as follows. We use
\begin{equation}\label{arct ob}
w:=u_n - \gamma(u_n-v)
\end{equation}
in \eqref{obst n}, where $\gamma (s)=\lambda\arctan(s/\lambda)$, for $\lambda>0$, and $v\in \mathcal K_\psi(\Omega)$ is arbitrary. Note that this is a legitimate test function, that is $w\in \mathcal K_\psi(\Omega)$. Indeed, on the set where $u_n\geqslant v$ we have
$\gamma(u_n-v)\leqslant u_n-v$ and so $w \geqslant v$; on the other hand, on the set where $u_n\leqslant v$ we have
$ \gamma  (u_n-v)\leqslant 0$ and so $w\geqslant u_n$. Therefore, from \eqref{obst n} we get
\begin{equation}\label{var in n}
\int_\Omega A_n(x,u_n,\nabla u_n) \nabla \gamma(u_n-v)\dx x \leqslant \left\langle \Phi , \gamma(u_n-v) \right \rangle\,.
\end{equation}
Following the lines of the proof of Theorem \ref{main} (where $\lambda=1$), we get  in turn \eqref{27}, \eqref{29.4} and finally \eqref{29.6}. To pass to the limit for fixed general $\lambda>0$ in \eqref{var in n}, we rewrite it as follows:
\begin{equation}\label{var in n 2}
\begin{split}
\int_\Omega [A_n(x,u_n,\nabla u_n)&-A_n(x,u_n,\nabla v)]\,\nabla \gamma(u_n-v)\dx x \\
&\leqslant \left\langle \Phi , \gamma(u_n-v) \right \rangle-\int_\Omega A_n(x,u_n,\nabla v)\nabla \gamma(u_n-v)\dx x \,.
\end{split}
\end{equation}
In the left hand side we use Fatou lemma, as by monotonicity condition \eqref{monotonia} the integrand is nonnegative. In the right hand side, we note that $A_n(x,u_n,\nabla v)\,\gamma'(u_n-v)$ converges to $A (x,u ,\nabla v)\,\gamma'(u-v)$ in $L^{p'}$, compare with \eqref{compattezza} where we did not use that $u_n\to u$. Hence, we deduce from \eqref{var in n 2} 
\[\begin{split}
\int_\Omega [A(x,u,\nabla u)&-A(x,u,\nabla v)]\,\nabla \gamma(u-v)\dx x \\
&\leqslant \left\langle \Phi , \gamma(u-v) \right \rangle-\int_\Omega A(x,u,\nabla v)\nabla \gamma(u-v)\dx x \,,
\end{split}
\]
that is
\begin{equation}\label{var in n 3}
\int_\Omega A(x,u,\nabla u)\,\nabla \gamma(u-v)\dx x\leqslant \left\langle \Phi , \gamma(u-v) \right \rangle\,.
\end{equation}
Now we let $\lambda\to\infty$ in \eqref{var in n 3}, noting that $\gamma(u-v)\to u-v$ strongly in $\SobpO$. Therefore, we get
\[\int_\Omega A(x,u,\nabla u)\,\nabla (u-v)\dx x \leqslant \left\langle \Phi , u-v \right \rangle\,,\]
for all $v\in K_\psi(\Omega)$, which means exactly that $u$ is a solution to our obstacle problem.
\end{proof}
\begin{remark}
Clearly, Theorem \ref{ob theorem} is more general than Theorem \ref{main} since we are allowed to choose $\psi \equiv -\infty$.
Indeed, in this case, the obstacle problem \eqref{obst} reduces to \eqref{Dirichlet}.
\end{remark}

\section{Regularity of the solution}\label{reg sol}
In this Section, following \cite{GMZ} we study regularity of the problem \eqref{Dirichlet}. 

\begin{theorem}
Let $1 < p<r<N$ and
$\Phi \in W^{-1, \frac {r}{p-1}  } (\Omega)$.
Assume that 
\eqref{coercivita},
\eqref{limitatezza} and
\eqref{monotonia}  hold  with
$\varphi \in L^r(\Omega)$. Under these hypotheses, there exists $\delta=\delta(\alpha, N,p,r)>0$ such that   
if
\begin{equation}\label{lndeboleinfty r} 
\dist_{L^{N,\infty}} (b,L^\infty (\Omega))<\delta,
\end{equation}
then any solution $u \in W^{1,p}_0(\Omega)$ of \eqref{Dirichlet} 
satisfies 
\begin{equation}\label{6.2}
|u|^{r^\ast/p^\ast} \in W^{1,p}_0(\Omega)
\end{equation}
In particular $u \in L^{r^\ast} (\Omega)$.
\end{theorem}
\begin{proof}
Let $u \in W^{1,p}_0(\Omega)$ be a solution 
of  \eqref{Dirichlet}. We may write $\Phi \in W^{-1,\frac {r}{p-1}} (\Omega)$ as
\[\Phi=\divergenza(|F|^{p-2} F)\]
for a suitable $F \in L^{r} (\Omega,\mathbb R^N)$.\par
For fixed $k>0$, we use $v:=u-T_k u$ as a test function  in  \eqref{soluzint} to get
\begin{equation}\label{202005221}
\alpha\int_{\Omega_k}|\nabla u|^p\dx x\leqslant \int_{\Omega_k}|F|^{p-1}\,|\nabla u|\dx x+\int_{\Omega_k}(b^p|u|^p+\fff^p)\dx x
\end{equation}
where $\Omega_k$ denotes the superlevel set $\{|u|>k\}$. For $0<\ee<\alpha$, by Young inequality we get
\begin{equation}\label{202005226}
(\alpha-\ee)\int_{\Omega_k}|\nabla u|^p\dx x\leqslant \int_{\Omega_k}(C\,|F|^p+b^p|u|^p+\fff^p)\dx x
\end{equation}
with $C=C(p,\ee)>0$. We let
\begin{equation}\label{202005227}
\lambda=\frac{r^*}{p^*}-1
\end{equation}
and multiply both sides of \eqref{202005226} by $k^{p\lambda-1}$ and integrate w.r.t.\ $k$ over the interval $[0,K]$, for $K>0$ fixed. By Fubini theorem we have
\begin{equation}\label{202005222}
(\alpha-\ee)\int_{\Omega}|\nabla u|^p\,|T_Ku|^{p\lambda}\dx x\leqslant \int_{\Omega}(C\,|F|^p+b^p|u|^p+\fff^p)\,|T_Ku|^{p\lambda}\dx x
\end{equation}
which implies
\begin{equation}\label{202005234}
(\alpha-\ee)^{\frac1p}\|\nabla u\,|T_Ku|^{\lambda}\|_p\leqslant C\|F\,|T_Ku|^{\lambda}\|_p+\|b\,u\,|T_Ku|^{\lambda}\|_p+\|\fff\,|T_Ku|^{\lambda}\|_p
\end{equation}
For $M>0$ we write
\begin{equation}\label{202005223}
\|b\,u\,|T_Ku|^\lambda\|_p\leqslant \|(b-T_Mb)\,u\,|T_Ku|^\lambda\|_p+M\|u\,|T_Ku|^\lambda\|_p
\end{equation}
By H\"older inequality and Sobolev embedding Theorem \ref{lorentz}
\begin{equation}\label{202005224}
\begin{array}{rl}
\ds \|(b-T_Mb)\,u\,|T_Ku|^\lambda\|_p\kern-.7em &\ds \leqslant \|b-T_M b\|_{N,\infty}\,\|u\,|T_Ku|^\lambda\|_{p^*,p}\vspace{10pt}\\
&\ds \leqslant \|b-T_M b\|_{N,\infty}\,S_{N,p}\,\|\nabla(u\,|T_Ku|^\lambda)\|_p
\end{array}
\end{equation}
Moreover,
\begin{equation}\label{202005225}
|\nabla(u\,|T_Ku|^\lambda)|\leqslant (1+\lambda)\,|\nabla u|\,|T_Ku|^\lambda
\end{equation}
Therefore
\begin{equation}\label{202005228}
\|(b-T_Mb)\,u\,|T_Ku|^\lambda\|_{L^p(\Omega)}\leqslant \|b-T_M b\|_{N,\infty}\,S_{N,p}\,(1+\lambda)\,\|\nabla u\,|T_Ku|^\lambda\|_p
\end{equation}
Under the assumption
\begin{equation}\label{202005229}
\|b-T_M b\|_{N,\infty}\,S_{N,p}\,(1+\lambda)<\alpha^{\frac1p}
\end{equation}
choosing $\ee$ small enough we get from \eqref{202005234}
\begin{equation}\label{2020052210}
\|\nabla u\,|T_Ku|^{\lambda}\|_p\leqslant C\,\|G\,|T_Ku|^{\lambda}\|_p
\end{equation}
with $C=C(p,r,M,\alpha)>0$, where we set
\begin{equation}\label{202005231}
G^p=|F|^p+|u|^p+\fff^p\,.
\end{equation}\par
We first show the claim under the additional assumption $u\in L^r(\Omega)$, so that $G\in L^r(\Omega)$. By H\"{o}lder inequality we have
\begin{equation}\label{I1}
\|G\,|T_Ku|^{\lambda}\|_p\leqslant
\|G\|_r\,\|T_K u\|^{\lambda}_{\lambda\,\frac {rp}{r-p}}
\end{equation}
From \eqref{202005227} we get
\begin{equation}\label{beta2}
\lambda\,\frac {rp}{r-p}=r^*\,.
\end{equation}
Hence, by Sobolev embedding theorem we have
\begin{equation}\label{I1bis}
\begin{array}{rl}
\ds \|T_K u\|^{\lambda}_{\lambda\,\frac {rp}{r-p}}\kern -.7em &\ds=\|T_K u\|^{\lambda}_{r^*}=\||T_K u|^{\frac{r^*}{p^*}}\|^{\lambda\frac{p^*}{r^*}}_{p^*}\leqslant C\,\|\nabla |T_K u|^{\frac{r^*}{p^*}}\|^{\lambda\frac{p^*}{r^*}}_{p}
\\
&\ds\leqslant C\,\|\nabla u\,|T_Ku|^{\lambda}\|_p^{\frac\lambda{\lambda+1}}
\end{array}
\end{equation}
Then,  combining \eqref{2020052210}, \eqref{I1} and \eqref{I1bis}, we get
\begin{equation}\label{5.7bis}
\|\nabla u\,|T_Ku|^{\lambda}\|_p^{\frac{p^*}{r^*}}\leqslant C\,\|G\|_r
\end{equation}
Passing to the limit as $K\rightarrow + \infty$ and recalling \eqref{202005231}, 
we have
\begin{equation}\label{202005232}
\|\nabla u\,|u|^{\lambda}\|_p^{\frac{p^*}{r^*}}\leqslant C\,(\|F\|_r+\|\fff\|_r+\|u\|_r)
\end{equation}
that is
\begin{equation}\label{202005233}
\|\nabla |u|^{\frac{r^*}{p^*}}\|_p\leqslant C\,(\|F\|_r+\|\fff\|_r+\|u\|_r)^{\frac{r^*}{p^*}}
\end{equation}
Hence, \eqref{6.2} holds as long as $u \in L^r(\Omega)$. At this point we observe that if 
$r\leqslant p^\ast$, using the
Sobolev embedding theorem,  $u\in L^{p^*}(\Om)$ and the proof is concluded.  
In the complementary case $r > p^\ast$, we use a bootstrap approach. Precisely, we repeat the previous 
argument replacing  $r$ with $p^\ast$ to get $u \in L^{p^{\ast\ast}}(\Omega)$. Using this information,  if  $r\leqslant p^{\ast\ast}  $, there is nothing left to prove. Otherwise we repeat previous argument again. In a finite number of  similar steps we can conclude our proof.
\end{proof}
\begin{remark}
In view of \eqref{202005229}, we may take
\[\delta=\frac {\alpha^\frac1p}{S_{N,p}}\,\frac{p^*}{r^*}\]
in \eqref{lndeboleinfty r}. Since $r\mappaa r^*$ is increasing, a similar condition to~\eqref{202005229} will hold in all intermediate steps, in case we need the bootstrap argument.  Note that $\delta$ reduces to the distance in \eqref{lndeboleinfty}, for $r=p$. 
\end{remark}



%
%



\end{document}